\documentclass{article}

\usepackage{fancyhdr}
\usepackage[utf8]{inputenc}    
\usepackage[T1]{fontenc}

\usepackage{amsmath,amssymb,amsthm}
\usepackage{tikz}
\usepackage{wrapfig} %
\usepackage{graphicx} %
\usepackage{tikz-cd}
\usepackage[margin=1.4in]{geometry} %
\usepackage[backend=biber, style=alphabetic]{biblatex}
\usepackage[colorlinks,linkcolor=darkblue,citecolor=darkblue,urlcolor=darkblue,breaklinks=true]{hyperref}
\usepackage{tipa} %
\usepackage{mathtools} %
\usepackage[capitalize]{cleveref} %

\usepackage{graphicx}
\DeclareMathOperator{\La}{\mathcal{L}}

\DeclareMathOperator{\Nb}{\mathbb{N}}

\DeclareMathOperator{\Rb}{\mathbb{R}}

\DeclareMathOperator{\join}{\vee}
\DeclareMathOperator{\meet}{\wedge}

\DeclareMathOperator{\id}{\mathsf{id}}

\DeclareMathOperator{\Prop}{\textbf{Prop}}

\newcommand{\type}[1]{\mathsf{{#1}}}
\newcommand{\term}[1]{\mathsf{{#1}}}

\newcommand{\true}{\mathtt{true}}
\newcommand{\false}{\mathtt{false}}

\newcommand{\rest}[2]{#1\big|_{#2}} %

\makeatletter
\providecommand*{\xmapstofill@}{%
  \arrowfill@{\mapstochar\relbar}\relbar\rightarrow
}
\providecommand*{\xmapsto}[2][]{%
  \ext@arrow 0395\xmapstofill@{#1}{#2}%
}

\makeatletter
\def\slashedarrowfill@#1#2#3#4#5{%
  $\m@th\thickmuskip0mu\medmuskip\thickmuskip\thinmuskip\thickmuskip
   \relax#5#1\mkern-7mu%
   \cleaders\hbox{$#5\mkern-2mu#2\mkern-2mu$}\hfill
   \mathclap{#3}\mathclap{#2}%
   \cleaders\hbox{$#5\mkern-2mu#2\mkern-2mu$}\hfill
   \mkern-7mu#4$%
}
\def\rightslashedarrowfill@{%
  \slashedarrowfill@\relbar\relbar\mapstochar\rightarrow}
\newcommand\xslashedrightarrow[2][]{%
  \ext@arrow 0055{\rightslashedarrowfill@}{#1}{#2}}
\makeatother

\newcommand{\allows}[2]{\Diamond_{#2}^{#1}}
\newcommand{\ensure}[2]{\Box_{#2}^{#1}}

\theoremstyle{definition}
\newtheorem{defn}{Definition}

\newtheorem{ex}[defn]{Example}

\newtheorem{rmk}{Remark}

\newtheorem{lemma}[defn]{Lemma}

\newtheorem{prop}[defn]{Proposition}

\definecolor{darkblue}{rgb}{0,0,0.7}

\title{Behavioral Mereology: A Modal Logic for Passing Constraints}
\author{Brendan Fong, David Jaz Myers, David I. Spivak}
\date{}

\begin{document}
\maketitle

\begin{abstract}
	Mereology is the study of parts and the relationships that hold between them. We introduce a behavioral approach to mereology, in which systems and their parts are known only by the types of behavior they can exhibit. Our discussion is formally topos-theoretic, and agnostic to the topos, providing maximal generality; however, by using only its internal logic we can hide the details and readers may assume a completely elementary set-theoretic discussion. We consider the relationship between various parts of a whole in terms of how behavioral constraints are passed between them, and give an inter-modal logic that generalizes the usual alethic modalities in the setting of symmetric accessibility.
\end{abstract}

\section{Introduction}

Many thinkers, from Heidegger to Isham and D\"{o}ring have asked ``What is a thing?'' \cite{heidegger1972thing, doring2010thing}. Heidegger for example says,
\begin{quote}
    From the range of the basic questions of metaphysics we shall here ask this one question: What is a thing? The question is quite old. What remains ever new about it is merely that it must be asked again and again.
\end{quote}
In this article, our way of asking about things is focused on the mereological aspect of things, i.e.\ the relationship between parts and wholes. The point of departure is that, at the very least, a part affects a whole: ``when you pull on a part, the rest comes with.'' For example, wherever my left hand is, my right hand is not far away. A whole then, has the property that it coordinates constraints---or said another way, it enables constraints to be \emph{passed}---between parts. In this article, we present a logic for constraint passing.

Our approach has roots in categorical logic, and in particular Lawvere's
observation that existential and universal quantification can be characterized
as adjoints to pullback, in any topos. In particular, a system, or more
evocatively a \emph{behavior type} $B_S$, will be a set which we imagine as the
set of ways a system can behave over time. If $S$ is a dynamical system, then we
may think of $B_S$ as the set of lawful trajectories of this system.\footnote{More generally, we may take a behavior type to be an object in
  any topos \cite{MacLane.Moerdijk:1992a}. This allows behavior types which
  where the behaviors may vary in time (or space). Since we don't expect our
  audience to know any topos theory, and since all the ideas we describe will
  make sense in any topos, we will use the language of sets through this paper,
  and leave experts to make the topos-theoretic translation themselves.} We are
inspired here by Willem's behavioral approach to control theory (see \cites{Willems:1998a}{Willems:2007a}{Willems.Polderman:2013a}).

We work behavior-theoretically; to paraphrase Gump, ``X is as X does''. We
associate to a system $S$ its set $B_S$ of possible behaviors --- $B_S$ is the
\emph{behavior type} of $S$. If $P$ is a part
of our system $S$, then if we know the total behavior of $S$ we also know the
behavior of $P$; so, we have a function $|_P : B_S \to B_P$ which we think of as
``restricting'' the behavior $b \in B_S$ of $S$ to the behavior $b|_P \in B_P$
of $P$. However, we are considering $P$ \emph{as a part of} $S$, so every
behavior of $P$ must come from some behavior of the whole system $S$; so, the
restriction map $\rest{}{P}\colon B_S \to B_P$ must be surjective.

We use this analysis of
parthood to \emph{define} a part of the system $S$ to be a quotient of $B_S$, i.e.\ a surjection $B_S \twoheadrightarrow B_P$ from the behavior type.
This surjection describes how a behavior of the entire system determines a behavior of the part, and any behavior of the part \emph{qua} part must extend to the behavior of the whole: for any behavior of my hand, there exists at least one compatible behavior of my whole body. Given a behavior on one part, we can consider all possible extensions to the whole, and subsequently ask how those extensions restrict to behaviors of other parts. In this way one part may constrain another.

To describe the logic of these constraints, we introduce two new logical
operators, or ``inter-modalities'', closely related to the classical ``it is
possible that'' and ``it is necessary that'' modalities, known as the
\emph{alethic} modalities \cite{kripke1963semantical}. We view a constraint
$\phi$ on a part $P$ as a predicate on the behaviors of $P$ --- the predicate
``satisfies the constraint $\phi$''. We may ask whether satisfying this
predicate allows, or whether it ensures, various behaviors on another part
$B_Q$: the constraint $\phi$ is passed in these two ways  from $P$ to $Q$.  To
be a bit more explicit, the first new operator is called \emph{allows}, written
$\allows{P}{Q}\phi$. This describes the set of behaviors in $B_Q$ for which
there exists an extension in $B_S$ that restricts to some behavior satisfying
$\phi$. The second is the operator \emph{ensures}, written $\ensure{P}{Q}\phi$.
This describes the set of behaviors in $B_Q$ for which all extensions to $B_S$
restrict to some behavior satisfying $\phi$ in $B_P$. Our goal in this paper is
to describe the properties of these inter-modalities (``inter'' because they go
from one part to another), and to demonstrate their utility with some elementary
examples.

Our inter-modalities allow us to faithfully express concepts of the behavioral
approach to control theory as expressed in Willems' \emph{Open Dynamical Systems
and their Control} \cite{Willems:1998a}. In particular:
\begin{itemize}
\item A time-invariant system $S$ is \emph{controllable} if and only if for any
  two behaviors $b_1$ and $b_2$, there is a time delay $D$ such that the behavior
  $b_1|_{<0}$ restricted to time before $0$ and $b_2|_{>D}$ restricted to time
  after $D$ are \emph{compatible} in the sense of Definition \ref{defn:compatible}.
\item If $b_1$ is a behavior of a part $P$ of $S$ and $b_2$ a behavior of part
  $Q$ of $S$, then $b_1$ is \emph{observable} from $b_2$ if and only if $b_1$
  \emph{determines} $b_2$ in the sense of Definition \ref{defn:determines}.
\item If $C$ is the constraint a controller $P$ places on the behavior of a plant
  $Q$, then the \emph{controlled behavior} is the constraint $\allows{P}{Q}C$ of
  behaviors of $Q$ which are \emph{allowed} by $C$ in the sense of Definition
  \ref{defn:intermodalities}.
\item The \emph{control problem} is the problem of choosing a constraint $C$ on
  $P$ so that a constraint $\phi$ of the plant $Q$ is satisfied. The universal
  solution to this problem is given by the constraint $\ensure{Q}{P} \phi$ of
  behaviors on $P$ which \emph{ensure} that $Q$ behaves according to $\phi$, in
  the sense of Definition \ref{defn:intermodalities}.
\end{itemize}
We believe the logic for constraint passing presented in this paper can be a
useful tool for formalizing arguments in the behavioral approach to control theory.

\section{Systems and Their Parts: Behavioral perspective}

When we constrain a part of a system, we are constraining \emph{what it does}. This suggests that we should model a system by its \emph{type of possible behaviors}, its \textbf{behavior type}.

\subsection{Systems as behavior types}

Luckily, we won't need to settle on what precisely a behavior type is, so long
as we can reason about it logically. For this, we need the behavior type of our
system and its parts to be objects in a topos; then, we can use the internal
logic of the topos to reason about our behavior types. A topos can be understood
as a system of \emph{variable sets}; in our case, this allows us the freedom to
have sets varying in time, in space, or according to different observers. We
will just work in the topos of sets and functions, leaving it to the experts to
extend the theory to arbitrary toposes.

So, a \emph{behavior type} is simply a set whose elements are regarded as possible behaviors of a system.
We can think of it as the ``phase space'' of our system, in a general sense.
This terminology is inspired by Willem's behavioral approach to control theory \cite{Willems.Polderman:2013a}, which
describes a dynamical system as a subset $B \subseteq W^T$ of lawful
trajectories parameterized by time $T$ in some value space $W$. The set $B$ is
the behavior type of this system, where a behavior of the system is simply a
lawful trajectory. Many of our examples follow this general form.

\begin{ex}
We will present a few running examples of systems considered in terms of their behavior types. Let's introduce them now.

\begin{itemize}
    \item Consider a bicycle. The bicycle pedals might be moving at some speed $p$, and the bicycle wheels might be moving at some speed $w$, both real numbers. If the pedal is pushing at a certain speed, then the wheels are moving at a least a constant multiple of that speed. Therefore, we will take the behavior type of our bicycle to be
    $$B_{\type{Bicycle}} := \{(p, w) \in \Rb \times \Rb \mid w \geq rp \}$$
    for some fixed ratio $r \in \Rb$.
    
    \item Consider a glass of water placed in a room of temperature $R$. The glass of water has temperature $T_t \in \Rb$ for every time $t \in \Nb$. By Newton's principle, the temperature of the water satisfies the following simple recurrence relation:
    $$T_{t + 1} = T_t + k(R - T_t).$$
    Therefore, the behavior type of this glass of water is
    $$B_{\type{Water}} := \{ T \in \Nb \to \Rb \mid T_{t + 1} = T_t + k(R - T_t)\}.$$
    
    \item Consider an ecosystem consisting of foxes and rabbits. At any given time $t \in \Nb$, there are $f_t$ foxes and $r_t$ rabbits, where we mask our uncertainty about the precise population by allowing these to be arbitrary real values, rather than integer values. The population of the species at time $t + 1$ is determined by its population at time $t$, according to the relation  $R_t(f, r)$ given by the following standard recurrences:
    \begin{align*}
        f_{t + 1} &= (1 - d_f)f_t + c_f r_t f_t, \quad\text{and} \\
        r_{t + 1} &= (1 + b_r)r_t - c_r r_t f_t.
    \end{align*}
    This is known as the Lotka--Volterra predator--prey model, but we could use any model. Here $d_f$ is the death rate of the foxes, $b_r$ is the birth rate of rabbits, and $c_f$ and $c_r$ are rates at which the consumption of rabbits by foxes affect their respective populations.   From here, we take the behavior type of this ecosystem to be
    $$B_{\type{Eco}} := \{(f, r) : \Nb \to \Rb \times \Rb \mid \forall t.\, R_t(f, r)\}.$$
    
    \item In algebraic geometry, we may think of the set of polynomial functions on a space $S$ as the possible behaviors of $S$. Thinking this way, the polynomial ring $\Rb[x, y]$ is the behavior type of the affine plane $\Rb \times \Rb$:
    $$B_{\Rb^2} := \Rb[x, y].$$
\end{itemize}
\end{ex}

\subsection{Parts as quotients of behavior type}
If we know a whole system $S$ (say, my body) is behaving like $b$, then we also know how any part $P$ of $S$ (say, my hand) is behaving: we just look at what $P$ is doing while $S$ does $b$. In other words, there should be a \emph{restriction} function, which we denote $\rest{}{P}\colon B_S \to B_P$, from the behavior type of the whole system to the behavior type of the part. 

Moreover, every behavior of a part $P$ will arise from \emph{some} behavior of the whole system: how could a part of the system do something if the system as a whole had no behaviors in which $P$ was doing that thing? Remember, we are considering the part $P$ \emph{as a part of the system $S$}, not on its own: my hand, not a severed hand. If we sever $P$ from the system $S$, it may be able to behave in ways that have no extension to $S$. But as a part of the system $S$, every behavior of the $P$ must be restricted from a behavior of $S$. We will give examples below, but first a definition.

\begin{defn}\label{def.part}
The behavior type of a \emph{part} of a system $S$ is a surjection
$\rest{}{P}\colon B_S \to B_P$ out of $B_S$ which we call the ``restriction from
$S$ to $P$''. We define the category of parts of $B_S$ to have as objects the parts of $S$ and as morphisms the commuting triangles
        \begin{center}
            \begin{tikzcd}[ampersand replacement=\&]
             \& B_S \arrow[ld, "|_P"', two heads] \arrow[rd, "|_Q", two heads] \&  \\
            B_P \arrow[rr, "|_Q"',two heads] \&  \& B_Q
            \end{tikzcd}
        \end{center}
Note that if such a map $B_P \to B_Q$ exists, then it will be unique and a
surjection. If there is such a map, we write $P \geq Q$ and say that $Q$ is a
part of $P$. This gives a preorder on parts, which we refer to as the
\emph{lattice of parts} of $S$.\footnote{We will see in Section
  \ref{subsec:compatibility.and.the.lattice.of.parts} that it is indeed a
  lattice.}
\end{defn}

For example, suppose that a system $S$ is divided into a \emph{plant} $P$ and a
\emph{controller} $C$. For example, $B_S$ might be $\{\alpha : \Rb \to \Rb^{p + c}
\mid \La(\alpha)\}$ a set of $p + c$ real variables satisfying a dynamical law $f$,
the first $p$ of which concern the plant $P$ and the last $c$ of which concern
the controller $C$. Then $B_P$ would be $\{\rho : \Rb \to \Rb^{p} \mid \exists
\gamma : \Rb \to \Rb^{c}.\, \La(\rho, \gamma)\}$ of $p$ real variables for which
there is some extension of $c$ variables (the behavior of the controller) which
is valid according to the dynamical law. The projection map $\Rb^{p + c} \to
\Rb^p$ gives a surjection from $B_S \twoheadleftarrow B_P$, witnessing that the
plant is a \emph{part} of the whole system.

In practice, we may have certain parts of $S$ in mind, and so we may consider a sublattice of that defined in \cref{def.part}. 

\begin{rmk}
\Cref{def.part} may look a little backwards. Usually a ``part'' is a subset; here we have defined a part to be a \emph{quotient}. What we have defined to be a part is often called a \emph{partition} (of $B_S$). 

What is happening here is a well-known ``space/function'' duality: we are not considering the system $S$ as some sort of object in space, but rather its type of behaviors $B_S$. Often, the behaviors $B_S$ of a system $S$ may be realized as functions on some sort of space $S$; this gives us a contravariance in $S$, which we see in the definition of part $Q$ is part of $P$ if there is a map ``going the other way,'' $B_P\to B_Q$. For comparison, consider the situation in algebraic geometry where one has a contravariant equivalence between the categories of algebraic varieties and the category of reduced finitely presented algebras; an algebra consists of polynomial functions on a variety, which tells us the ``possible behaviors'' of this variety, relative to the base field.
\end{rmk}

\begin{ex}
Here are some examples of parts.
\begin{itemize}
    \item The two parts of the bicycle under consideration are the pedal and the wheel. Explicitly, the behavior types of these parts of the bicycle are the types of all possible behaviors which arise as some behavior of the whole bicycle:
    \begin{align*}
        B_{\type{Pedal}} &:= \{ p \in \Rb \mid \exists w. (p, w) \in B_{\type{Bicycle}}\},\\
        B_{\type{Wheel}} &:= \{ w \in \Rb \mid \exists p. (p, w) \in B_{\type{Bicycle}}\}.
    \end{align*}
    In this case, every real number is a possible speed of the pedal, and every real number a possible speed of the wheel.
    
    \item In the system $\type{Water}$ of the cup of water sitting in the room, there is just one thing we are considering the behavior of: the cup. But, we can see this behavior at many different times. For every time $t \in \Nb$, we get a part $\type{Water}_t$ of the cup at time $t$ with behaviors
    $$B_{\type{Water}_t} := \{ x \in \Rb \mid \exists T \in B_{\type{Water}}.\, T_t = x\}.$$
    In fact, for any set $D \subseteq \Nb$ of times, we get the behavior type of the cup during $D$:
    $$B_{\type{Water}_D} := \{ x \in D \to \Rb \mid \exists T \in B_{\type{Water}}.\, \forall d \in D.\, T_d = x_d\}.$$
    
    \item The ecosystem consisting of foxes and rabbits is more complicated than the cup, but the principle is the same. We can consider the system at different times, and restrict our attention to just foxes or rabbits as we please. In particular, we let $\type{Fox}_t$ be the system of foxes at time $t$, and $\type{Rabbit}_t$ be the system of rabbits at time $t$. These have behavior types
    \begin{align*}
        B_{\type{Fox}_t} &:= \{ x \in \Rb \mid \exists (f, r) \in B_{\type{Eco}}.\, f_t = x\},\\
        B_{\type{Rabbit}_t} &:= \{ x \in \Rb \mid \exists (f, r) \in B_{\type{Eco}}.\, r_t = x\}.
    \end{align*}
    
    \item In algebraic geometry, the parts of the plane $\Rb^2$ under consideration are the algebraic subsets, i.e.\ the subsets carved out by algebraic equations. If $f \in \Rb[x, y]$ is a polynomial, then $\Rb[x, y]/(f)$ is the algebra of polynomial functions where $f$ is $0$. This is the algebra of polynomial functions on the subspace $\{(x, y) \in \Rb^2 \mid f(x, y) = 0\}$ of the plane. In the case that $f = x^2 + y^2 - 1$, then this subspace is the circle, and so we see that the behavior type of the circle is the quotient
    $$B_{\type{Circle}} := \Rb[x, y]/(x^2 + y^2 - 1).$$
    
\end{itemize}
\end{ex}

A surjection $\rest{}{P}\colon B_S\to B_P$ out of a set may equally be presented by its kernel pair, the equivalence relation on $B_S$ where $b\sim_P b'$ iff $\rest{b}{P}=\rest{b'}{P}$. We may call this relation \emph{observational equivalence}; the behaviors $b,b'$ with $b\sim_Pb'$ are \emph{observationally equivalent} relative to $P$. This is clearest when thinking of $P$ as some measuring device in a larger system; two behaviors of the whole system are observationally equivalent relative to our measuring device when it measures them to be the same. Two behaviors of my body are hand-equivalent if they are indistinguishable by looking at my hand; and two times are clock-equivalent when they read the same on the clock face.

There is essentially (i.e.\ up to isomorphism) no distinciton between a quotient of $B_S$ and an equivalence relation on $B_S$. Thus we are defining parts of $S$ to be equivalence relations on $S$-behaviors. This seems to be a novel approach to mereology, though we cannot claim to know the literature well enough to be sure.

\subsection{Compatibility}

Now we turn our attention to how behaviors of various parts of the system relate to one another. The most basic relation between behaviors of two parts is that of being simultaneously realizable by a behavior of the whole system. We call this relation \emph{compatibility}.

\begin{defn}\label{defn:compatible}
If $P$ and $Q$ are parts of $S$, then we say behaviors $a \in B_P$ and $b \in B_Q$ are \emph{compatible}, denoted $\mathfrak{c}(a, b)$, if there is a behavior $s \in B_S$ of the whole system that restricts to both $a$ and $b$, i.e.\
$$\mathfrak{c}(a, b) :\equiv \exists s \in S.\, a = \rest{s}{P}\, \wedge\, \rest{s}{Q} = b.$$

Generally, if $a_i \in P_i$ is some family of behaviors indexed by a set $I$, then this family is said to be \emph{compatible} if there is an $s \in S$ such that $\rest{s}{P_i} = a_i$ for all $i\in I$.
\end{defn}

In other words, two behaviors (one on each of two parts) are compatible if there is a behavior of the whole system that restricts to both of them. 

\begin{ex}
Examples of compatible behaviors are easily obtained by restricting a single system behavior to two parts.
\begin{itemize}
\item In the bicycle example, we see that a speed $p$ of the pedal is compatible with a speed $w$ of the wheel if and only if $w \geq rp$:
$$\mathfrak{c}(p, w) = w \geq rp.$$

\item In the cup of water example, a temperature $T^0 \in B_{\type{Water}_t}$ at time $t$ is compatible with a temperature $T^1 \in B_{\type{Water}_{t'}}$ at a later time $t'$ are compatible if and only if $T^1$ follows from $T^0$ via the recurrence relation. In particular, if $t' = t + 1$, then
$$\mathfrak{c}(T^0, T^1) = (T^1 = T^0 + k(R - T^0)).$$

\item In the ecosystem example, we have a number a different comparisons to choose from. A fox population $f^0 \in B_{\type{Fox}_t}$ at time $t$ is compatible with $f^1 \in B_{\type{Fox}_{t + 1}}$ at time $t + 1$ if and only if there is simultaneous rabbit population $r^0$ so that $f^1 = (1 - d_f)f^0 + c_f r^0 f^0$. 

Two simultaneous fox and rabbit populations are compatible if and only if there is some history of the ecosystem which achieves those population at that time. In particular, any two populations of foxes and rabbits at time $0$ are compatible.

\item In the algebraic geometry example, a part is an algebraic subset of the plane and its behaviors are the polynomial functions defined there. Two behaviors are compatible iff there is a function on the plane to which both simultaneously extend. For example, if $p(x)$ is a behavior on the $x$-axis and $q(y)$ is a behavior on the $y$-axis, and we have $p(0)=q(0)$, then writing $a\coloneqq p(0)\in\Rb$,  the behavior $p(x)+q(y)-a$ on the whole plane  restricts to each, and thus $p$ and $q$ are compatible. This is obviously a necessary condition: if $p(0)\neq q(0)$ then no global function can restrict to both.
\end{itemize}
\end{ex}

\subsection{Compatibility and the lattice of parts}\label{subsec:compatibility.and.the.lattice.of.parts}
We can express some of the parts-lattice operations in terms of the compatibility relation.

\begin{prop}
The meet $P \cap Q$ of parts $P$ and $Q$ of $S$ has behavior type given by the following pushout.
\[
\begin{tikzcd}
 & B_S \arrow[ld, two heads] \arrow[rd, two heads] &  \\
B_P \arrow[rd, two heads] &  & B_Q \arrow[ld, two heads] \\
 & B_{P \cap Q} \ar[uu, phantom, very near start, "\rotatebox{45}{$\urcorner$}"]& 
\end{tikzcd}
\]
In other words, a behavior of $P \cap Q$ is either a behavior of $P$ or a behavior of $Q$, where these are considered equal if they are compatible.
$$B_{P \cap Q} \cong \frac{B_P + B_Q}{\mathfrak{c}}.$$

Here, $B_P + B_Q$ is the disjoint union of these two sets, and we are
quotienting out by the smallest equivalence relation for which $p \sim q$
whenever $\mathfrak{c}(p, q)$.

Dually, the join $P \cup Q$ has behaviors given by the image factorization of the induced map $B_S \to B_P \times B_Q$.
\[
\begin{tikzcd}
 & B_S \arrow[ld, two heads] \arrow[rd, two heads]\ar[d, two heads] &  \\
B_P & B_{P\cup Q}\ar[l, two heads]\ar[r, two heads]\ar[d, tail] & B_Q  \\
 & B_P\times B_Q\ar[ul, two heads]\ar[ur, two heads] 
\end{tikzcd}
\]
In other words, a behavior of $P \cup Q$ is a pair of compatible behaviors from $P$ and from $Q$.
$$B_{P \cup Q} \cong \{(a, b) \in B_P \times B_Q | \mathfrak{c}(a, b)\}.$$

Furthermore, the largest part $\top$ is $S$, and the smallest part $\bot$ is
empty if $S$ is empty and a singleton otherwise.
\end{prop}
\begin{proof}
It is straightforward to show that the universal property of the meet and that of the pushout are equivalent, and similarly for the top element $\top$. The other two claims (joins and bottom elements) are proved by an orthogonality argument; we focus on the join.

If $B_S\twoheadrightarrow B_R$ is a part with $B_R\geq B_P$ and $B_R\geq B_Q$, we need to show $B_R\geq B_{P\cup Q}$ as defined above. But under these hypotheses there exists a solid-arrow a square
\[
\begin{tikzcd}
    B_S\ar[r, two heads]\ar[d, two heads]&
    B_{P\cup Q}\ar[d, tail]\\
    B_R\ar[r]\ar[ur, dotted]&
    B_P\times B_Q
\end{tikzcd}
\]
and there exists a dotted lift as shown by the following argument.\footnote{Or,
  one may note that epis are orthogonal to monos in a topos.} Suppose that
$r \in B_R$ and let $s \in B_S$ be an element such that $\rest{s}{R} = r$. Then
by commutativity of the diagram, $(\rest{r}{P}, \rest{r}{Q}) = (\rest{s}{P},
\rest{s}{Q})$ so that $\rest{r}{Q}$ and $\rest{r}{P}$ are compatible (being
restrictions of $s$). So $(\rest{r}{P},\rest{r}{Q}) \in B_{P \cup Q}$, giving
the dotted lift above.
\end{proof}

\begin{defn}
A part $P$ is \emph{strongly disjoint} from a part $Q$ if every behavior of $P$
is compatible with every behavior of $Q$. The two parts $P$ and $Q$ are
\emph{disjoint} if their intersection $P \cap Q$ is the minimal part. Strongly
disjoint parts are disjoint:
$$\forall a \in B_P.\, \forall b \in B_Q.\, \mathfrak{c}(a, b) \quad\Rightarrow\quad B_{P \cap Q} = \bot$$
\end{defn}

In general, we will be more interested in joins than in meets because joins are easier to work with (being subsets of a product, rather than quotients of a disjoint union by a non-transitive relation).

\begin{ex}
Let's consider some examples of joins and meets of parts.
\begin{itemize}
    \item In the example of the bicycle, note that we have
    $$B_{\type{Bicycle}} = B_{\type{Pedal} \cup \type{Wheel}},$$
    since a behavior of the bicycle was defined precisely to be a behavior of a pedal and a wheel satisfying a compatibility constraint.
    
    \item In the example of the cup of water, the behaviors $B_{\type{Cup}_D}$ over a duration $D \subset \Nb$ of times are the union of the behaviors $B_{\type{Cup}_d}$ for each time $d \in D$:
    $$B_{\type{Cup}_D} = B_{\bigcup \type{Cup}_d}.$$
    
    \item Similarly, in the ecosystem example, the parts of the ecosystem at various times are the join of parts at particular times. More interestingly, recall that every behavior of $\type{Fox}_0$ (starting population of foxes) is compatible with every behavior of $\type{Rabbit}_0$ (starting population of rabbits). Therefore, 
    $$B_{\type{Fox}_0 \cap \type{Rabbit}_0} = \bot$$
    This witnesses the fact that the parts $\type{Fox}_0$ and $\type{Rabbit}_0$ do not at all mutually constrain each other, and so have no shared sub-parts.
    
    \item Consider two algebraic subsets $X$ and $Y$ of the plane and let $I$ and $J$ be the ideals of polynomials which vanish on $X$ and $Y$ respectively. Then we know that $B_X = \Rb[x, y]/I$ and $B_Y = \Rb[x, y]/J$. The join has behaviors 
    \begin{align*}
        B_{X \cup Y} &= \{(f, g) \in \Rb[x, y]/I \times \Rb[x, y]/J \mid \exists h \in \Rb[x, y].\, h + I = f + I \mbox{ and } h + J = g + J\} \\
            &= \{(h + I, h + J) \mid h \in \Rb[x, y]\}.
    \end{align*}
    We note that $B_{X \cup Y}$ is a ring with operations taken componentwise and that the restriction map $B_{\Rb^2} \to B_{X \cup Y}$ is a homomorphism. The kernel is the set of those $h \in \Rb[x, y]$ for which $h + I = I$ and $h + J = J$, i.e. precisely those $h \in I \cap J$. Therefore, 
    $$B_{X \cup Y} = \Rb[x, y]/(I \cap J),$$
    which represents the union in the usual geometric sense.  
    
\end{itemize}
\end{ex}

\subsection{Determination recovers the order of parts}
\begin{defn}\label{defn:determines}
If $P$ and $Q$ are parts of $S$, and $a \in B_P$ and $b \in B_Q$, then $a$ \emph{determines} $b$ if every behavior $s$ of the whole system $S$ which restricts to $a$ also restricts to $b$.
$$\mathfrak{d}(a, b) :\equiv \forall s \in B_S.\, \rest{s}{P} = a \Rightarrow \rest{s}{Q} = b.$$
We say that a part $P$ \emph{determines} a part $Q$ if every behavior $a \in B_P$, determines some behavior $b \in B_Q$. 
\end{defn}

This is a much stronger notion than compatibility, and we shall show in \cref{Things:lem:Determines.Mereology} that it can be used to recover the original order relation $\geq$ on parts.

\begin{lemma}
A behavior always determines uniquely: if $\mathfrak{d}(a, b)$ and $\mathfrak{d}(a, b')$, then $b = b'$.
\end{lemma}
\begin{proof}
We know there is some $s \in S$ which restricts to $a$. Since $a$ determines $b$ and $b'$, $s$ restricts to both $b$ and $b'$; but then $b = b'$.
\end{proof}

\begin{prop}\label{Things:lem:Determines.Mereology}
For parts $P$ and $Q$ of $S$, the following are equivalent:
\begin{enumerate}
    \item $Q$ is a part of $P$, i.e.\ there is a surjection $B_P\twoheadrightarrow
      B_Q$ under $B_S$.
    \item $P$ $\mathfrak{c}$-determines $Q$, in the sense that for every $a \in B_P$ there is a unique $b \in B_Q$ such that $a$ is compatible with $b$.  In other words, $\forall a \in B_P.\, \exists! b \in B_Q.\, \mathfrak{c}(a, b)$.
    \item $P$ $\mathfrak{d}$-determines $Q$, in the sense that for every $a \in B_P$ there is a $b \in B_Q$ such that $a$ determines $b$. In other words, $\forall a \in B_P.\, \exists b \in B_Q.\, \mathfrak{d}(a, b)$.
    \item For all $a \in B_P$ and $b \in B_Q$, if $a$ is compatible with $b$, then $a$ determines $b$.
\end{enumerate}
\end{prop}
\begin{proof}
This claim follows from the following implications.

($1 \Rightarrow 2$): Recall that if $Q$ is a part of $P$, then the restriction of $s \in B_S$ to $B_P$ and then $B_Q$ is equal to the restriction $f \in B_P \twoheadrightarrow B_Q$ straight to $B_Q$. This immediately implies $a \in B_P$ is compatible with $f(a)$. For uniqueness, notice that if $b \in B_Q$ is also compatible with $a$, then some $s \in B_S$ restricts to both $a$ and $b$. But $s$ restricts also to $f(a)$, so $b = f(a)$.
    
($2 \Rightarrow 3$): Suppose $P$ $\mathfrak{c}$-determines $Q$. Then for $a \in B_P$, let $b \in B_Q$ be the guaranteed unique compatible behavior of $Q$. If $s \in B_S$ restricts to $a$, then $a$ is compatible with $\rest{s}{Q}$ and so $\rest{s}{Q} = b$; therefore, $a$ $\mathfrak{d}$-determines $b$.
    
($3 \Rightarrow 1$): By the lemma, $a$ determines $b$ uniquely; thus we get a function $f : B_P \to B_Q$ sending $a$ to $f(a) := b$. Now, for any $b \in B_Q$, there is some $s \in B_S$ restricting to it. By hypothesis, $\rest{s}{P}$ determines some $f(a) \in B_Q$; but then $\rest{s}{Q} = f(a)$ so that $b = f(a)$ and $f$ is epi. Finally, if $s \in B_S$, then $f(\rest{s}{P}) = \rest{s}{Q}$, so that $Q$ is part of $P$.
    
($3 \Leftrightarrow 4$): Since $P$ and $Q$ are both parts of $S$, for every $a
\in B_P$, there is a $b \in B_Q$ compatible with it. Thus $4 \Rightarrow 3$. On the other hand, assuming $3$ and that $a$ is compatible with $b$, we know by $2$ that $b$ is the unique behavior of $Q$ compatible with $a$, so that $a$ determines $b$. \qedhere
\end{proof}

\begin{ex}
Here are some examples of one part determining another; equivalently, these are examples of subparts of parts.
\begin{itemize}
    \item In the example of the bicycle, neither the wheel nor the pedal determines the other.
    \item In the example of the cup of water, each temperature at time $t$ determines the temperature at time $t + 1$ by the recurrence relation. Therefore, we have $\type{Cup}_t \geq \type{Cup}_{t + 1}$ by Lemma \ref{Things:lem:Determines.Mereology}; in other words, a cup at time $t + 1$ is ``part'' of a cup at time $t$. This may seem odd, but remember that our notion of part is behavior-theoretic. In a deterministic world, the future is \emph{contained in}---part of---the present. This can be seen from the perspective of Laplace's demon: the future is ``present before its eyes''; it is part of the present.
    \item Since the ecosystem example is also a deterministic dynamical system, the analysis is the same as for the cup of water example.
    \item For the algebraic geometry example, the behavior of the unit circle determines the behavior of the part $\{(1,0)\}$, but it does not determine the behavior of the $x$-axis because the function $x^2+y^2$ and the constant function $1$ have the same behavior on the circle but different behaviors on the $x$-axis.
\end{itemize}
\end{ex}

\section{Constraints, Allowance, and Ensurance}

In this section, we introduce our new logical operators, $\allows{}{}$ and $\ensure{}{}$, and prove some basic properties about them. We shall see in \cref{subsec.compat_ensure} that these two operators pass constraints between parts. But first, what is a constraint?

\subsection{Constraints as predicates}
We will identify a constraint $\phi$ on a part $P$ with the predicate ``satisfies $\phi$'' on behaviors $B_P$ of $P$. In other words, we have the following definition.

Let $\Prop$ be the two element set $\{\true,\false\}$ of truth values. We think
of functions $\phi : X \to \Prop$ as predicates concerning the elements of $X$
--- applied to $x \in X$, $\phi$ gives a truth value $\phi(x)$ which says
whether or not $x$ satisfies the predicate $\phi$. 
\begin{defn}
A \emph{constraint} on a part $P$ is a map $\phi : B_P \to \Prop$. The type of constraints on $P$ is $\Prop^P$. We write
$$\phi \vdash \psi$$
to mean that $\phi$ \emph{entails} $\psi$, that is, if $\phi(b) = \true$, then $\psi(b) = \true$.
\end{defn}

For parts $P\geq Q$, we get an \emph{adjoint triple} that allows us to transform
constraints on $P$ to those on $Q$, and vice versa, given by the logical quantifiers:
    \begin{center}
        \begin{tikzcd}[ampersand replacement=\&]
        \Prop^{B_P} \arrow[rr, "\exists^P_Q"{name=exists}, bend left] \arrow[rr, "\forall^P_Q"'{name=forall}, bend right] \&  \& \Prop^{B_Q} \arrow[ll, "\Delta^Q_P" description, ""{name=delta}]
        \ar[from=exists, to=delta, phantom, "\Rightarrow"]
        \ar[from=delta, to=forall, phantom, "\Leftarrow"]
        \end{tikzcd}
    \end{center}
These functors are defined logically as follows:
\begin{align*}
    \exists^P_Q \phi(q) &:= \exists p \in B_P.\, \big((\rest{p}{Q} = q) \meet \phi(p)\big) \\
    \Delta^Q_P \psi(p) &:= \psi(\rest{p}{Q}) \\
    \forall^P_Q \phi(q) &:= \forall p \in B_P.\, \big((\rest{p}{Q} = q) \Rightarrow \phi(p)\big)
\end{align*}
The fact that they are \emph{adjoint} means that
\begin{align*}
    \exists^P_Q \phi \vdash \psi\quad &\iff \quad\phi \vdash \Delta^Q_P \psi \\
    \Delta^Q_P \psi \vdash \xi\quad &\iff\quad \psi \vdash \forall^P_Q \xi 
\end{align*}

We will write $\exists_P$, $\Delta^P$, and $\forall_P$ for $\exists^S_P$, $\Delta^P_S$ and $\forall^S_P$ respectively. As mentioned, these operations are functorial, meaning that $\exists^P_P(\phi)=\Delta^P_P(\phi)=\forall^P_P(\phi)=\phi$ and for $R \leq Q \leq P$,
\begin{align*}
    \exists^Q_R  \exists^P_Q &= \exists^P_R, \\
    \Delta^R_Q \Delta^Q_P &= \Delta^R_P, \\
    \forall^Q_R \forall^P_Q  &= \forall^P_R.
\end{align*}

\begin{lemma}\label{Thing:lem:Global.Modality.Props}
Recall that for part $P$ of system $S$ we write $s \sim_P s'$ for the relation $\rest{s}{P} = \rest{s'}{P}$ on $B_S$. Then for any predicate $\phi$ on $S$ we have:
\begin{enumerate}
    \item $\Delta^P \exists_P \phi(s) = \exists s'.\, (s\sim_P s') \wedge \phi(s')$
    \item $\Delta^P \forall_P \phi(s) = \forall s'.\, (s \sim_P s') \Rightarrow \phi(s')$
    \item $\phi \vdash \Delta^P \exists_P \phi$\quad and \quad $\Delta^P \forall_P \phi \vdash \phi$
\end{enumerate}
\end{lemma}
Thinking again of $P$ as a way to observe behaviors,
$\Delta^P \exists_P \phi$ is the set of system behaviors $s$ that our observer says plausibly satisfy $\phi$: there is something $P$-equivalent to $s$ that satisfies $\phi$. And $\Delta^P \forall_P \phi$ is the set of system behaviors that our observer can guarantee satisfy $\phi$.

\subsection{The allowance and ensurance operators}\label{subsec.compat_ensure}
Now we turn to the question of how constraints on the behavior of some part of the system constrain the behavior of other parts. We discuss two ways to pass constraints between parts.

     \begin{defn}\label{defn:intermodalities}
        A constraint $\phi$ on a part $P$ induces a constraint on a part $Q$ (of
        the same system $S$) in two universal ways:
        \begin{itemize}
            \item ``Allows $\phi$'':\quad $\allows{P}{Q}\phi := \exists_Q \Delta^P\phi$
                \begin{align*}
               \allows{P}{Q} \phi(q) &=  \exists s \in B_S.\, (\rest{s}{Q} = q) \wedge \phi(\rest{s}{P})\\
               &= \exists p \in B_P.\, \mathfrak{c}(p, q) \meet \phi(p).
                \end{align*}
            \item ``Ensures $\phi$'':\quad $\ensure{P}{Q}\phi := \forall_Q  \Delta^P\phi$
                \begin{align*}
                \ensure{P}{Q} \phi(q) &= \forall s \in B_S.\, (\rest{s}{Q} = q) \Rightarrow \phi(\rest{s}{P})\\
                &= \forall p \in B_p.\, \mathfrak{c}(p, q) \Rightarrow \phi(p).
                \end{align*}
        \end{itemize} 
    \end{defn}
    
    A behavior $q$ of $Q$ allows a constraint $\phi$ on $P$ if $Q$ can be doing $q$ while $P$ is satisfying $\phi$; we write this as $\allows{P}{Q}\phi(q)$. A behavior $q$ of $Q$ ensures $\phi$ on $P$ if whenever $Q$ does $q$, $P$ \emph{must} satisfy $\phi$; we write this as $\ensure{P}{Q}\phi(q)$.
    
    These symbols are chosen due to their relation to the usual modalities of possibility ($\Diamond$) and necessity ($\Box$) \cite{kripke1963semantical}; a behavior $q$ allows $\phi$ if it is \emph{possible} that $P$ satisfies $\phi$ while $Q$ does $q$, and a behavior $q$ ensures $\phi$ if it is \emph{necessary} that $P$ satisfies $\phi$ while $Q$ does $q$. Indeed, in the case that the accessibility relation in the Kripke frame is an equivalence relation, we will be able to recover the usual possibility and necessity modalities from our allowance and ensurance operators (see \cref{subsec.modal}).
    
    Note that compatibility and determination appear as particular, pointwise
    cases of the allowance and ensurance operators. For any $p\in B_P$ and $q\in
    B_Q$, we have
    \begin{align*}
        \mathfrak{c}(p, q) &= \allows{P}{Q}(=p)(q) = \allows{Q}{P}(=q)(p) \\
        \mathfrak{d}(p, q) &= \ensure{Q}{P}(=q)(p)
    \end{align*}
    We write $(=p)$ for the map $B_P\to\Prop$ that sends $p'$ to $\true$ if and only if $p=p'$.
    
    \begin{ex}
    We return to our running examples to see our new operators in action.
    \begin{itemize}
        \item In the example of the bicycle with gear ratio of $r$, we can ask what behavior of the pedal is ensured by the wheels moving slower than $w=2$ mph. We have $\ensure{\type{Pedal}}{\type{Wheel}}(\leq 2)$ is the constraint $p\leq \frac{2}{r}$.
        
        \item If the cup of water has temperature $T^0\in \type{Water}_0$ at time $0$, then it cannot have a temperature further away from the ambient room temperature $R$ at a later time. Therefore, 
        $$|R - T^t| > |R - T^0| \vdash \neg \allows{\type{Water}_0}{\type{Water}_t}(=T^0)(T^t).$$
        
        \item Suppose that in the ecosystem example, one was given the goal of introducing a fox population at time $0$ in order to keep the rabbit population in check after a given deadline $d$. Let's say that being kept in check means being between two fixed bounds, $$r_t \mapsto \term{inCheck}(r_t) := k_1 < r_t < k_2$$
        so that $\term{inCheck} : B_{\type{Rabbit}_t} \to \Prop$ is a constraint on rabbits at time $t$. The constraint of being kept in check for all times after the deadline $d$ is the constraint
        $$r \mapsto \forall t \geq d.\, \term{inCheck}(r_t)$$
        on the join $\bigvee_{t \geq d} \type{Rabbit}_t$. The goal may then be expressed as finding a starting fox population $f_0$ which ensures that the rabbit population is kept in check at all times after the deadline:
        $$\ensure{\bigvee_{t \geq d} \type{Rabbit}_t}{\type{Fox}_0}(\forall t \geq d.\, \term{inCheck})(f_0).$$   
        \end{itemize}
 \end{ex}
 
\begin{ex}
We can see a higher-order ensurance in the ecosystem example. If there are any rabbits at time $0$, and if the rabbit population is bounded independent of time, then the rabbits must ensure that there are foxes, and that the foxes ensure there are rabbits:
    $$r_0 \geq 0 \meet r < k \vdash \ensure{F}{R} (f > 0 \meet \ensure{R}{F}(r > 0)).$$ 
If there are no foxes, then the rabbit population is unbounded, and if there are foxes, then there must be rabbits for them to eat. We see that this ecosystem model exhibits a rudimentary form of symbiosis; though the foxes eat the rabbits, they counter-intuitively must ensure that the rabbits do not go extinct, lest they themselves go extinct.
\end{ex}

    Assuming the law of excluded middle, our operators are inter-definable by conjugating with negation. 
    
    \begin{prop}\label{prop.deMorgan}
    Assuming Boolean logic, allowance and ensurance are de Morgan duals. That is, $\neg \allows{P}{Q} \neg = \ensure{P}{Q}$.
    \end{prop}
    \begin{proof} The proof uses the law of excluded middle twice:
    \begin{align*}
        \neg \allows{P}{Q} \neg \phi(q) &= \neg \exists p.\, \mathfrak{c}(p, q) \meet \neg \phi(p) \\
            &= \forall p.\, \neg (\mathfrak{c}(p, q) \meet \neg \phi(p)) \\
            &= \forall p.\, \neg \mathfrak{c}(p, q) \join \neg\neg \phi(p) \\
            &= \forall p.\, \mathfrak{c}(p, q) \Rightarrow \phi(p). \qedhere
    \end{align*}
    \end{proof}
    Note that \cref{prop.deMorgan} does not generalize to arbitrary toposes, where the variation of the sets (in time or in space) means that one must reason constructively in general.
    
\subsection{Allowance and ensurance as adjoints}
We now develop the basic theory of the allowance and ensurance operators. First, we observe that allowance and ensurance are monotone and adjoint to each other.   
   
    \begin{prop} \label{Things:prop:compatible.and.ensures} Let $P$ and $Q$ be parts of the system $S$ and $\phi$ a constraint on $P$. Then:
    \begin{enumerate}
        \item If a constraint $\phi$ entails $\psi$, then allowing $\phi$ entails allowing $\psi$, and ensuring $\phi$ entails ensuring $\phi$. That is, $\allows{P}{Q}$ and $\ensure{P}{Q}$ are monotone.
        \item If $q$ ensures that $P$ does $\phi$, then $q$ allows $P$ doing $\phi$. That is, $\ensure{P}{Q}\phi \vdash \allows{P}{Q}\phi$.
        \item Allowing $\phi$ entails $\psi$ if and only if $\phi$ entails ensuring $\psi$. That is, $\allows{P}{Q}$ is left adjoint to $\ensure{Q}{P}$.
        \item $\allows{P}{Q}$ commutes with $\join$ and $\exists$, and $\ensure{P}{Q}$ commutes with $\meet$ and $\forall$.
        \item A constraint allows itself if and only if it is satisfied if and only if it ensures itself. In other words, $\allows{P}{P}\phi = \phi = \ensure{P}{P}\phi$.

    \end{enumerate}
    \end{prop}
    \begin{proof}
    ~
    \begin{enumerate}
        \item As the composite of monotone maps, both maps are monotone.
        \item Let $\phi : B_P \to \Prop$ be a constraint on $P$. Suppose that $\ensure{P}{Q}\phi(q)$ for $q \in Q$, that is that for all $s \in B_S$, if $\rest{s}{Q} = q$, then $\phi(\rest{s}{P})$. Since $\rest{}{Q}$ is epi, there is an $s$ such that $\rest{s}{Q} = q$, and therefore $\phi(\rest{s}{P})$; so $\ensure{P}{Q}\phi(q) \vdash \allows{P}{Q}\phi(q)$.
        \item This follows from the (left, right) adjoint pairs $(\exists_Q, \Delta^Q)$ and $(\Delta^P, \forall_P)$ as follows:
        \begin{align*}
            \allows{P}{Q} \phi &\vdash \psi \\ 
            \exists_Q \Delta^P \phi &\vdash \psi \\
            \Delta^P \phi &\vdash \Delta^Q \psi \\ 
            \phi &\vdash \forall_P \Delta^Q \psi \\
            \phi &\vdash \ensure{Q}{P} \psi.
        \end{align*}
        \item Follows from the general properties of adjoints.
        \item As $\allows{P}{P}$ and $\ensure{P}{P}$ are adjoint, it suffices to show that just one is the identity. This follows from the fact that $P$ is a quotient of $S$. Suppose $\ensure{P}{P}\phi(p)$; that is, for all $s \in S$, $\rest{s}{P} = p$ implies $\phi(p)$. Since $\rest{}{P}$ is epi, there is some $s$ that restricts to $p$, and therefore $\phi(p)$. \qedhere
        
    \end{enumerate}
    \end{proof}
    
 To make the adjointness of $\allows{P}{Q}$ and $\ensure{Q}{P}$ more visceral, consider the following example: by putting my left hand on the wall, I can \emph{ensure} my right hand is within ten feet of the wall. This is the same as saying that if the behavior of my right hand \emph{allows} my left hand being on the wall, then my right hand is within 10 feet of the wall. That is, $L\vdash\ensure{Q}{P}R$ iff $\allows{P}{Q}L\vdash R$.
    
    The unit and counit of this adjunction are interesting as well. The unit
    says that whatever my left hand is doing, doing this ensures that my right
    hand allows doing it. The counit says that my right hand allowing my left hand ensuring that my right hand is doing something implies
    that my right hand is doing that thing. For example, if my right hand allows
    my left hand to ensure that it is within 10 feet of the wall, then my right hand must be within 10 feet of the wall.
    
\subsection{Interactions between three parts}
    Next we discuss how allowance and ensurance compose, something like a triangle inequality for these inter-modalities.
    \begin{prop}[Composition]
    Let $P$, $Q$, and $R$ be parts of system $S$. Then:
    \begin{enumerate}
        \item If a behavior of $R$ allows a constraint on $P$, then it allows a behavior of $Q$ allowing that constraint. That is, $\allows{P}{R} \vdash \allows{Q}
        {R}  \allows{P}{Q}$.
        \item If a behavior of $R$ ensures that a behavior of $Q$ ensures some constraint on $P$, then that behavior of $R$ ensures that constraint on $P$. That is, $\ensure{Q}{R}\ensure{P}{Q} \vdash \ensure{R}{P}$.
    \end{enumerate}
    \end{prop}
    \begin{proof}
    This follows immediately from \cref{Thing:lem:Global.Modality.Props}, since 
    \[
    \allows{P}{R} = \exists_R \Delta^P \vdash \exists_R \Delta^Q \exists_Q \Delta^P = \allows{Q}{R}\allows{P}{Q}
    \]
    \[
    \ensure{Q}{R}\ensure{P}{Q} = \forall_R \Delta^Q \forall_Q \Delta^P \vdash \forall_R \Delta^P = \ensure{P}{R} \qedhere
    \]
    \end{proof}

\begin{prop}[Interaction of inter-modalities with meets and joins]
~
\begin{enumerate}
    \item $\Diamond^P_{Q \cap R}\phi = \exists q \in Q,\, r \in R.\, \mathfrak{c}(q,r) \wedge \Diamond^P_{Q \cup R}\phi(q, r)$.
    \item $\Diamond^P_{Q \cup R} \phi(q, r) \vdash \Diamond^P_Q \phi(q) \wedge \Diamond^P_R \phi(r)$.
    \item $\Box^P_{Q \cap R}\phi = \forall q \in Q,\, r \in R.\, (\mathfrak{c}(q,r) \Rightarrow \Box^P_{Q \cup R}\phi(q, r)$).
    \item $\Box^P_Q \phi(q) \vee \Box^P_R \phi(r)  \vdash  \Box^P_{Q \cup R} \phi(q, r)$.
\end{enumerate}
\end{prop}
\begin{proof}
We prove the first two; the remaining two are dual.
\begin{enumerate}
    \item ($\Rightarrow$) Suppose that $\allows{P}{Q \cap R}\phi(a)$. Then there is an $s \in B_S$ so that $\rest{s}{Q \cap R} = a$ and $\phi(\rest{s}{P})$. But then $\rest{s}{Q}$ and $\rest{s}{R}$ are compatible and $\rest{s}{Q \cup R} = (\rest{s}{Q},\, \rest{s}{R})$, so that $\allows{P}{Q \cup R}\phi(\rest{s}{Q},\, \rest{s}{R})$.
    
    ($\Leftarrow$) Suppose that there are compatible $q \in Q$ and $r \in R$ with $\allows{P}{Q \cup R}\phi(q, r)$. Since $q$ and $r$ are compatible, $\rest{q}{Q \cap R} = \rest{r}{Q \cap R}$, so that $\allows{P}{Q \cap R}\phi(\rest{q}{Q \cap R})$.
    \item This is mostly a matter of unpacking definitions:
        \begin{align*}
            \allows{P}{Q \cap R}\phi(q, r) &= \exists s.\, (\rest{s}{Q \cup R} = (q, r)) \meet \phi(\rest{s}{P}) \\
                &= \exists s.\, (\rest{s}{Q} = q) \meet (\rest{s}{R} = r) \meet \phi(\rest{s}{P}) \\
                &\vdash (\exists s.\, (\rest{s}{Q} = q) \meet \phi(\rest{s}{P})) \meet (\exists s.\, (\rest{s}{R} = r) \meet \phi(\rest{s}{P})) \\
                &= \allows{P}{Q}\phi(q) \wedge \allows{P}{R}\phi(r). \qedhere
        \end{align*} 
\end{enumerate}
\end{proof}

\subsection{Allowance and ensurance between a part and subpart}
When $P$ determines $Q$, or equivalently $Q$ is a part of $P$ (see \cref{Things:lem:Determines.Mereology}), we have the following additional properties.
\begin{prop}[Allowance and ensurance of a part] \label{Things:prop:parts}
Suppose that $P \geq Q$. Then
\begin{align*}
    \allows{P}{Q} &= \exists^P_Q = \exists p.\, (\rest{p}{Q} = q) \meet \phi(p) \\
    \ensure{P}{Q} &= \forall^P_Q = \forall p.\, (\rest{p}{Q} = q) \Rightarrow \phi(p) \\
    \allows{Q}{P} &= \Delta^Q_P = \ensure{Q}{P}
\end{align*}
Conversely, if there are parts $P$, $Q$ of $S$, such that $\allows{Q}{P} \vdash \ensure{Q}{P}$, then $P \geq Q$.
\end{prop}
\begin{proof}
Note that as $S \twoheadrightarrow P$ is epi, $\exists^S_P\Delta^P_S = \id^P_P$. By definition, $\allows{P}{Q} = \exists_Q \Delta^P = \exists^P_Q \exists_P \Delta^P = \exists^P_Q$, and dually. Similarly, $\allows{Q}{P} = \exists_P\Delta^Q = \exists_P \Delta^P \Delta^Q_P = \Delta^Q_P$, and dually. For the converse, note that if $\allows{Q}{P} \vdash \ensure{Q}{P}$, then in particular $\allows{Q}{P}(=q)(p) \vdash \ensure{Q}{P}(=q)(p)$, or in other words $\mathfrak{c}(p, q) \vdash \mathfrak{d}(p, q)$; by  \cref{Things:lem:Determines.Mereology} this is equivalent to $P \geq Q$.
\end{proof}

Thus allowance and ensurance in the case of a part is simple: $q \in B_Q$ allows a constraint $\phi$ on $P$ if there exists $p \in B_Q$ that restricts to $q$ and $\phi(p)$, and $q$ ensures $\phi$ if all $p$ that restrict to $q$ obey $\phi$. On the other hand, $p \in B_Q$ allows a constraint $\psi$ on $Q$ iff $p$ ensures $\psi$ iff $\rest{p}{Q}$ obeys $\psi$, and in fact this gives another characterization of the notion of part.

We may make slightly more general statements as follows.
\begin{prop}\label{Things:prop:parts.operations}
Suppose that $P \leq P'$ and $Q' \leq Q$. Then
\begin{enumerate}
    \item $\allows{P'}{P} \allows{Q}{P'} = \allows{Q}{P}$.
    \item $\ensure{P'}{P} \ensure{Q}{P'} = \ensure{Q}{P}$.
    \item $\allows{P}{Q} \Delta^{Q}_{Q'} = \allows{P}{Q'}$.
    \item $\ensure{P}{Q} \Delta^{Q}_{Q'} = \ensure{P}{Q'}$.
\end{enumerate}
\end{prop}
\begin{proof}
This follows from the functoriality of $\exists$, $\forall$, and $\Delta$.
\end{proof}

\begin{ex}
In the ecosystem example, we were trying to ensure that rabbits were kept in check after the deadline $d$:
$$\ensure{\bigvee_{t \geq d }\type{Rabbit}_t}{\type{Fox}_0}(\forall t \geq d.\, \term{inCheck}(r_t)).$$
Using the above lemmas, we can change this into a simpler form. First, since ensurance is a right adjoint (\cref{Things:prop:compatible.and.ensures} item 3), we may rewrite our constraint as
\[
\forall t \geq d.\, \ensure{\bigvee_{t \geq d} \type{Rabbit}_t}{\type{Fox}_0}\term{inCheck}(r_t) 
\]
and then, because $\term{inCheck}(r_t)$ is more explicitly $\Delta_{\bigvee_{t \geq d} \type{Rabbit}_t}^{\type{Rabbit}_t}\term{inCheck}$ we may use \cref{Things:prop:parts.operations} item 4 to rewrite this as
\[
\forall t \geq d.\, \ensure{\type{Rabbit}_t}{\type{Fox}_0}\term{inCheck}(r_t) \]
In particular, we see that the suitability of a given initial fox population for keeping the rabbits in check may be determined by considering each time $t \geq d$ independently.
\end{ex}

\subsection{Allowance and ensurance, possibility and necessity}\label{subsec.modal}

Finally, we describe the manner in which our intermodalities generalize the classical alethic modalities of possibility and necessity.

Fix a part $P$. Then for any part $Q$, we obtain two modalities on $P$ by composing our intermodalities from $P$ to $Q$ with their corresponding intermodality from $Q$ to $P$. 
\begin{prop}
The operators $\allows{Q}{P}\allows{P}{Q}$ and $\ensure{Q}{P}\ensure{P}{Q}$ are a pair of adjoint modalities on $\Prop^{B_P}$. They are the identity modality if and only if $P \leq Q$.
\end{prop}
\begin{proof}
By \cref{Things:prop:compatible.and.ensures} item 3, these two modalities are the composites of adjoint pairs of operators, and hence are adjoint themselves. Moreover, $\allows{Q}{P}\allows{P}{Q}$ is the identity if and only if $\allows{Q}{P}\allows{P}{Q} \phi \vdash \phi$ for all $\phi$, which occurs if and only if $\allows{P}{Q}\phi \vdash \ensure{P}{Q}\phi$ for all $\phi$, which occurs if and only if $P \leq Q$ (by \cref{Things:prop:parts}).
\end{proof}

These modalities describe constraints on $P$ as seen through the part $Q$. To
obtain a description of possibility and necessity, assume that $B_S$ is inhabted
--- that there is some behavior of the system. We let $Q=\bot$ be the system whose behavior type $B_Q=\ast$ consists of just a single element. Then the adjoint modalities
$$\allows{\bot}{P}\allows{P}{\bot} \qquad \textrm{left adjoint to} \qquad \ensure{\bot}{P}\ensure{P}{\bot}$$
describe possibility and necessity on $B_P$. 

For example, for any constraint $\varphi$ on $B_P$, the constraint $\allows{\bot}{P}\allows{P}{\bot}(\varphi)$ maps all elements of $B_P$ to $\true$ if there is some behavior $p$ that satisfies $\varphi$, and maps all elements to $\false$ otherwise. Thus the modality detects whether $\varphi$ is \emph{possible}: that is, whether there is some behavior that satisfies $\varphi$. 

On the other hand, $\ensure{\bot}{P}\ensure{P}{\bot}(\varphi)$ is the constraint that maps all elements of $B_P$ to $\true$ if all behaviors $p \in B_P$ satisfy $\varphi$, and maps all elements to $\false$ otherwise; thus this modality detects whether $\varphi$ is always satisfied, and hence \emph{necessary}.

More generally, the usual semantics of the ``it is possible that'' and ``it is necessary that'' modalities $\Diamond$ and $\Box$ take place in a Kripke frame $(W,A)$, where $W$ is a set, known as the set of worlds, and $A$ is a binary relation on $W$ known as the accessibility relation. The predicate then $\Diamond(\varphi)(w)$ holds if $\varphi(w')$ for \emph{some} $w'$ such that $wAw'$, and $\Box(\varphi)(w)$ holds if $\varphi(w')$ for \emph{all} $w'$ such that $wAw'$ \cite{kripke1963semantical}. If $A$ is an equivalence relation, then we may equivalently describe $A$ by an epi $W \twoheadrightarrow W/A$. In this case, we have $\Diamond=\allows{W/A}{W}\allows{W}{W/A}$ and $\Box= \ensure{W/A}{W}\ensure{W}{W/A}$ as modalities on $\Prop^W$. 

\section{Outlook and Conclusion}
We have presented a logic that describes how constraints---restrictions on behavior---are passed from one part of a system to another. While we have presented this from a set theoretic point of view, we have taken care to use arguments that are valid in any topos (with the noted exception of \cref{prop.deMorgan}, which only holds in boolean toposes). As a consequence, our logic retains its character as a logic of constraint passing across a wide variety of semantics. One possibly valuable notion of semantics is one that captures a notion of time.

Indeed, behavior is best conceived as occurring over time, though of course the question of what time \emph{is} remains an issue. One can imagine that a system has, for any interval or ``window'' of time, a set of possible behaviors, and that each such behavior can be cropped or ``clipped'' to any smaller window of time. This is the perspective of \emph{temporal type theory} \cite{Schultz.Spivak:2017a}. While that work uses topos theory in a significant way, the main idea is easy enough.

Whether we speak of a bicycle, an ecosystem, or anything else that could be said to exist in time, it is possible to consider the set of behaviors of that thing over an interval of time, say over the ten-minute window $(0,10)$. Above we often discussed an idea which can be generalized to any system $S$ that exists in time. Namely, we can consider different parts of time as parts of $S$. Given any behavior $s$ over the 10-minute window, we can clip it to the first minute $\rest{s}{(0,1)}$; this gives a function $S(0,10)\to S(0,1)$, which is often called \emph{restriction}, though we will continue to call it clipping. Let's assume that every possible behavior at $(0,1)$ extends to some behavior over the whole interval---i.e.\ that the universe doesn't just end under certain conditions on $(0,1)$-behavior---at which point we have declared that the clipping function is surjective, and hence gives a part in the sense of section \cref{def.part}. We call it a \emph{temporal part}.

What then does it mean to pass constraints between temporal parts? The idea
begins to take on a control-theoretic flavor: behavioral constraints at one time
window may allow or ensure constraints at other time windows. A mother could say
``doing this now ensures no dessert tonight''. The child could ask ``does our
position on the road now allow me to play with Rutherford this afternoon?'' A control system could attempt to solve the problem ``what values of parameter $P$ can I choose, 5 seconds from now, that both allow current conditions and ensure that in 10 minutes we will achieve our target?''

In any case, as mentioned in the introduction, our original goal was to understand what makes a thing a thing, e.g.\ what gives things like bricks the quality of being cohesive (not two bricks) and closed (not the left half of a brick). We believe that a good logic of constraint passing between parts is essential for that, but perhaps not sufficient. The question of what additional structures need to be added or considered in order construct a viable notion of thing, remains future work.

\printbibliography
\end{document}